\documentclass[11pt]{amsart}
\usepackage{amsmath}
\usepackage{amssymb}
\usepackage{amscd}

\def\NZQ{\mathbb}               

\def\QQ{{\NZQ Q}}
\def\ZZ{{\NZQ Z}}
\def\RR{{\NZQ R}}

\newtheorem{Theorem}{Theorem}[section]
\newtheorem{Lemma}[Theorem]{Lemma}
\newtheorem{Corollary}[Theorem]{Corollary}
\newtheorem{Proposition}[Theorem]{Proposition}

\newtheorem{Example}[Theorem]{Example}

\newtheorem{Definition}[Theorem]{Definition}

\let\epsilon\varepsilon
\let\phi=\varphi
\let\kappa=\varkappa

\def \s {\sigma}

\begin{document}

\title{A generalization of the Abhyankar Jung theorem to associated graded rings of valuations}
\author{Steven Dale Cutkosky}
\thanks{Partially supported by NSF}

\address{Steven Dale Cutkosky, Department of Mathematics,
University of Missouri, Columbia, MO 65211, USA}
\email{cutkoskys@missouri.edu}

\begin{abstract} Suppose that $R\rightarrow S$ is an extension of local domains and $\nu^*$ is a valuation dominating $S$. We consider the natural extension of associated graded rings along the valuation ${\rm gr}_{\nu^*}(R)\rightarrow {\rm gr}_{\nu^*}(S)$. We give examples showing that in general, this extension does not share good properties of the extension $R\rightarrow S$, but after enough blow ups above the valuations, good properties of the extension  $R\rightarrow S$ are reflected in the extension of associated graded rings. Stable properties of this extension (after blowing up) are much better in characteristic zero than in positive characteristic. Our main result is a generalization of the Abhyankar-Jung theorem which holds for extensions of associated graded rings along the valuation, after enough blowing up.
\end{abstract}

\maketitle
\section{Introduction}
Suppose that  $K\rightarrow K^*$ is a finite field extension, $\nu^*$ is a valuation of $K^*$ and $\nu=\nu^*|K$. We will consider local rings $R$  of $K$ dominated by $\nu$ (we do not require local rings to be Noetherian) and the local ring $S$ of $K^*$ which is the localization of the integral closure of $R$ in $K^*$ at the center of $\nu^*$ (notation is given in Section \ref{SecNoc}). This extension $R\rightarrow S$ occurs in local uniformization with $R$ regular as a first step in reducing the multiplicity of $S$.

Let $e=[\Gamma_{\nu^*}:\Gamma_{\nu}]$ be the reduced ramification index of $\nu^*$ over $\nu$ and $f=[V_{\nu^*}/m_{\nu^*}:V_{\nu}/m_{\nu}]$ be the residue degree of the 
valuation rings $V_{\nu^*}$ of $\nu^*$ and $V_{\nu}$  of $\nu$ and $\delta(\nu^*/\nu)$ be the defect of $\nu^*$ over $\nu$ (these concepts are reviewed in Section \ref{SecNoc}).

Bernard Teissier has defined the associated graded ring along a valuation of a ring $R$ which is contained in a valuation ring in his work on resolution of singularities \cite{T}, \cite{T2} as follows. For $\gamma\in \Gamma_{\nu}$, let $\mathcal P_{\gamma}(R)=\{f\in R\mid \nu(f)\ge \gamma\}$ and
$\mathcal P_{\gamma}^+=\{f\in R\mid \nu(f)>\gamma\}$. The associated graded ring of $R$ along $\nu$ is defined as
$$
{\rm gr}_{\nu}(R)=\bigoplus_{\gamma\in \Gamma_{\nu}}\mathcal P_{\gamma}(R)/\mathcal P_{\gamma}^+(R).
$$
This ring is almost always non Noetherian. The quotient field ${\rm QF}({\rm gr}_{\nu}(R))={\rm QF}({\rm gr}_{\nu}(V_{\nu}))$ (Lemma \ref{LemmaQ1}) and 
$[{\rm QF}({\rm gr}_{\nu^*}(S)):{\rm QF}({\rm gr}_{\nu}(R))]=ef$ (Proposition \ref{PropQ3}).

In this paper, we show that after possibly replacing $R$ with  a sequence of blow ups along the valuation, the extension ${\rm gr}_{\nu}(R)\rightarrow {\rm gr}_{\nu^*}(S)$ shares many good properties with the extension $R\rightarrow S$. As is to be expected, stable properties under blowing up in characteristic zero are much better than in positive characteristic (compare Example \ref{Ex3} and Theorem \ref{Theorem1*}).

We begin by giving four examples where $R$ is a regular local ring and $\nu^*$ has rank 1, illustrating possible behavior of these extensions.
Example \ref{Ex4} (in Section \ref{SecA}), which is valid over any field of characteristic $\ne 2$, is of a regular local ring $R$ such that $S$ is not regular and ${\rm gr}_{\nu^*}(S)={\rm gr}_{\nu}(R)$. This behavior is caused by the fact that the rank of a valuation dominating a local ring may increase when extending the valuation to the completion of the ring (\cite{Sp}, \cite{HS}, \cite{C}, \cite{GAST}). If there are immediate extensions of $\nu$ and $\nu^*$ to
valuations dominating $\hat R$ and $\hat S$, then we have that $\hat R=\hat S$ if ${\rm gr}_{\nu^*}(S)={\rm gr}_{\nu}(R)$ (Proposition \ref{Prop42}) so the pathology of Example \ref{Ex4} cannot hold in this case.

The next three examples show that the extension ${\rm gr}_{\nu}(R)\rightarrow {\rm gr}_{\nu^*}(S)$ does not in general reflect good properties of the extension $R\rightarrow S$, so blowing up along the valuation is necessary to obtain the results of this paper.

 Example \ref{Ex1} is an example  over any field $k$, which shows that it is possible for ${\rm gr}_{\nu^*}(S)$ to not be integral over ${\rm gr}_{\nu}(R)$.
 
 \begin{Example}\label{Ex1} Let $k$ be a field, $K=k(u,y)$ and $K^*=k(x,y)$ be rational function fields over $k$ with an inclusion $K\rightarrow K^*$ induced by the substitution $u=yx+x^2$. $K^*$ is Galois over $K$ with Galois group $\ZZ_2$ if $\mbox{char}(k)\ne 2$.
 Let $R=k[u,y]_{(u,y)}$  and $S=k[x,y]_{(x,y)}$.  We have that $S$ is the integral closure of $R$ in $K^*$.  Define a rank 1 valuation $\nu^*$ on $K^*$ which dominates $S$ by prescribing that $\nu^*(x)=1$ and $\nu^*(y)=\pi$.  We have semigroups
 $$
 S^S(\nu^*)=\{\nu^*(f)\mid f\in S\setminus\{0\}\}=\ZZ_{\ge 0}+\ZZ_{\ge 0}\pi
 $$
 and 
 $$
 S^R(\nu)=\{\nu^*(f)\mid f\in R\setminus\{0\}\}=\ZZ_{\ge 0}(1+\pi)+\ZZ_{\ge 0}\pi.
 $$
 We have that $n\nu^*(x)\not \in S^R(\nu)$ for all $n\in \ZZ_{>0}$, so ${\rm gr}_{\nu^*}(S)$ is not integral over ${\rm gr}_{\nu}(R)$.
 \end{Example}
 
Example \ref{Ex2} (from \cite{CV1}), which is valid over any field, shows that it is possible for both $R$ and $S$ to be regular and the extension of associated graded rings to be integral but not finite. 

\begin{Example}\label{Ex2}  
Let $k$ be a field, $K=k(u,v)$ and $K^*=k(x,y)$ be rational function fields over $k$ with an inclusion $K\rightarrow K^*$ induced by the substitutions $u=x^2$ and $v=y^2$. $K^*$ is Galois over $K$ with Abelian Galois group if $\mbox{char}(k)\ne 2$.
 Let $R=k[u,v]_{(u,v)}$  and $S=k[x,y]_{(x,y)}$.  We have that $S$ is the integral closure of $R$ in $K^*$.  By  Example 9.4 \cite{CV1},
 there is a valuation $\nu^*$ of $K^*$ with value group $\Gamma_{\nu^*}=\frac{1}{3^{\infty}}\ZZ$ and $V_{\nu^*}/m_{\nu^*}=k$ such that ${\rm gr}_{\nu^*}(S)$ is integral over ${\rm gr}_{\nu}(R)$ but is not finite.
 \end{Example}

Example \ref{Ex3} (from \cite{CP}) is an example of an immediate extension (so that ${\rm gr}_{\nu^*}(S)$ and ${\rm gr}_{\nu}(R)$ have the same quotient fields) in characteristic $p>0$ such that ${\rm gr}_{\nu^*}(S)$ is integral (purely inseparable) but not finite over ${\rm gr}_{\nu}(R)$. In fact, ${\rm gr}_{\nu^*}(S)^p\subset {\rm gr}_{\nu}(R)$. In this example  the property  of being not finite and purely inseparable is stable under sequences of monoidial transforms along $\nu^*$.

\begin{Example}\label{Ex3} 
Let $k$ be a field of characteristic $p>0$, $K=k(u,v)$ and $K^*=k(x,y)$ be rational function fields over $k$ with an inclusion $K\rightarrow K^*$ induced by the substitutions $u=\frac{x^p}{1-x^{p-1}}$ and $v=y^p-x^{p-1}y$. $K^*$ is a tower of two Artin-Schreier extensions over $K$. 
 Let $R=k[u,v]_{(u,v)}$  and $S=k[x,y]_{(x,y)}$.  We have that $S$ is the integral closure of $R$ in $K^*$.  In  the main example of Section 7.11  \cite{CP},
  a valuation $\nu^*$ of $K^*$ is constructed which dominates $S$. The value groups are $\Gamma_{\nu^*}=\frac{1}{p^{\infty}}\ZZ=\Gamma_{\nu}$ and $V_{\nu^*}/m_{\nu^*}=k$ so that $\nu^*/\nu$ is an immediate extension. Further, $\nu^*$
  is the unique extension of $\nu=\nu^*|K$ to $K^*$. Thus the defect $\delta(\nu^*/\nu)=2$. Since the extension is immediate, we have that
  ${\rm QF}({\rm gr}_{\nu^*}(S))={\rm QF}({\rm gr}_{\nu}(R))$. In \cite{CP}, the associated graded rings are calculated as
  $$
  {\rm gr}_{\nu}(R)=k[U_0,U_1,\ldots]/(U_1^{p^2}-U_0,\{U_j-U_0^{p^{2j-2}}U_{j-1}\}_{2\le j})
  $$
  and
  $$
  {\rm gr}_{\nu^*}(S)=k[X_0,X_1,\ldots]/(X_1^{p^2}-X_0,\{X_j-X_0^{p^{2j-2}}X_{j-1}\}_{2\le j})
  $$
  and the inclusion ${\rm gr}_{\nu}(R)\rightarrow {\rm gr}_{\nu^*}(S)$  of $k$-algebras is obtained by the substitutions $U_j=X_j^p$ for $0\le j$. In particular, ${\rm gr}_{\nu^*}(S)$ is integral (even purely inseparable) over ${\rm gr}_{\nu}(R)$ but is not finite. We have that
  ${\rm gr}_{\nu^*}(S)^p\subset {\rm gr}_{\nu}(R)$. 
  
  It is shown in Section 7.11 of \cite{CP} that this property (of the extension of associated graded rings being integral but not finite) is stable under sequences of monoidal transforms 
  $$
  \begin{array}{cccc}
  R'&\rightarrow &S'&\subset V_{\nu^*}\\
  \uparrow&&\uparrow&\\
  R&\rightarrow& S&
  \end{array}
  $$
  above $R$ and $S$ which are dominated by $\nu^*$.
  
  \end{Example}

The conclusions of the following theorem should be compared with Examples \ref{Ex1} - \ref{Ex3}.

\begin{Theorem}\label{Theorem5*} Suppose that $K$ is an algebraic function field over an arbitrary field $k$ and $K^*$ is a finite extension of $K$. Suppose that $\nu^*$ is a rank one $k$-valuation of $K^*$.
 Let $\nu$ be the restriction of $\nu^*$ to $K$. 
 Then there exist a finite set of elements $f_1,\ldots,f_n\in V_{\nu}$ such that if 
$R$ is a local ring of $K$ which is dominated by $\nu$ which contains $f_1,\ldots,f_n$ and $S$  is the localization at the center of $\nu^*$ on the integral closure of $R$ in $K^*$, then
 ${\rm gr}_{\nu^*}(S)$ is integral over ${\rm gr}_{\nu}(R)$.
 \end{Theorem}
 
 Theorem \ref{Theorem5*} is proven in Section \ref{SecA}.

The conclusions of the following theorem should be compared with Example \ref{Ex3}. Example \ref{Ex3} is a tower of two immediate Artin-Schreier extensions and $\nu^*$ is the unique extension of $\nu$,
so the following theorem ensures that ${\rm gr}_{\nu^*}(S)^{p^2}\subset {\rm gr}_{\nu}(R)$ in Example \ref{Ex3}.

\begin{Theorem}\label{Theoreminsep*} Suppose that $K$ and $K^*$ are fields, $\nu^*$ is a valuation of $K^*$ with restriction $\nu$ to $K$, $K^*$ is Galois over $K$, $\nu^*$ has rank 1, $\nu^*$ is the unique extension of $\nu$ to $K^*$,
$[\Gamma_{\nu^*}:\Gamma_{\nu}]$ is a power of $p$ where $p$ is the residue characteristic of $V_{\nu*}$ and $V_{\nu^*}/m_{\nu^*}$ is purely inseparable over $[V_{\nu}/m_{\nu}]$
(so $[K^*:K]=p^n$ for some $n$). Suppose $R$ is a normal local ring of $K$ which is dominated by $\nu$ and $S$ is the localization of the integral closure of $R$ in $K^*$ at the center of $\nu^*$. Then ${\rm gr}_{\nu^*}(S)^{p^n}\subset {\rm gr}_{\nu}(R)$.
\end{Theorem}

Theorem \ref{Theoreminsep*} is proven in Section \ref{SecA}.

The classical Abhyankar Jung Theorem \cite {J}, \cite{RAF}, tells us that if $R$ is regular of equicharacteritic zero and the discriminant of $R$ in $K^*$ is a simple normal crossing divisor, then $S$ has only Abelian quotient singularities. A few references on this topic and related problems are \cite{GM}, \cite{Lu}, \cite{PG} and \cite{V}.

Our most difficult result is the following generalization of the Abhyankar Jung Theorem to associated graded rings of valuations.
Theorem \ref{Theorem1*} is proven in Section \ref{SecB}.

\begin{Theorem}\label{Theorem1*} Suppose that $K$ is an algebraic function field over an algebraically closed field $k$ of characteristic zero and $K^*$ is a finite extension of $K$. Suppose that $\nu^*$ is a rank one $k$-valuation of $K^*$
whose residue field is $k$. Let $\nu$ be the restriction of $\nu^*$ to $K$, and suppose that $R'$ is an algebraic local ring of $K$ which is dominated by $\nu$. Then there exists a sequence of monoidal transforms $R'\rightarrow R$ along $\nu$ such that $R$ is regular, and if $S$ is the local ring of the center of $\nu^*$ on the integral closure of $R$ in $K^*$, then 
\begin{enumerate}
\item[1)] ${\rm gr}_{\nu^*(S)}$ is a free ${\rm gr}_{\nu}(R)$-module of finite rank $e=[\Gamma_{\nu^*}/\Gamma_{\nu}]$.
\item[2)] $\Gamma^*/\Gamma$ acts on ${\rm gr}_{\nu^*}(S)$ with ${\rm gr}_{\nu^*}(S)^{\Gamma_{\nu^*}/\Gamma_{\nu}}\cong {\rm gr}_{\nu}(R)$.
\end{enumerate} 
\end{Theorem}

Theorem \ref{Theorem1*} is proven in dimension 2 (and rational rank 1) by Ghezzi, Ha and Kashcheyeva \cite{GHK} and the conclusions of 1) of the theorem are established by Ghezzi and Kascheyeva \cite{GK} for two dimensional defectless extensions of positive characteristic algebraic function fields. 
The conclusions of 1) of the theorem are established for excellent local domains $R'$ of dimension two under a defectless extension in \cite{C2}.
Thus it is reasonable to ask if the conclusions of 1) of the theorem are true for defectless extensions (assuming resolution of singularities is true). 

The assumption that $\nu$ has rank 1 is used in Theorem 1 to avoid some problems with extensions of valuations to the completion of an analytically irreducible local ring. The theorem could be true for arbitrary rank valuations. 

Examples \ref{Ex1} and \ref{Ex2} satisfy the classical discriminant condition of the Abhyankar Jung Theorem, but do not satisfy the conclusions of the theorem.
However, after some blowing up along the valuation,
$$
\begin{array}{cccc}
R_1&\rightarrow& S_1&\subset V_{\nu^*}\\
\uparrow&&\uparrow\\
R&\rightarrow&S
\end{array}
$$
$R_1\rightarrow S_1$ must satisfy the conclusions of the theorem (if the characteristic of the ground field $k$ is zero).

The characteristic $p>0$ example Example \ref{Ex3} is much worse. Finiteness of the extension of graded rings never holds after blowing up.

The key point  in the proof of Theorem \ref{Theorem1*} is to find $R\rightarrow S$ such the $\hat S=\bigoplus_{i=1}^e w_i\hat R$ where $\{\nu^*(w_i)\}$ is a complete set of representatives of the cosets of $\Gamma_{\nu}$ in $\Gamma_{\nu^*}$.

In general, $\hat S$ is not a free $\hat R$-module if $R$ is regular and $\dim R>2$, although the discriminant condition of the classical Abhyankar-Jung Theorem ensures this. 

The proof of Theorem \ref{Theorem1*} uses the local monomialization theorem \cite{C}, an extension in \cite{CG} giving  nice extensions of the valuations to the completions of the local rings and the classical Abhyankar Jung Theorem \cite{RAF}.

\section{Notation and Preliminaries}\label{SecNoc}
\subsection{Local algebra.}  All rings will be commutative with identity. A ring $S$ is essentially of finite type over $R$ if $S$ is a local ring of a finitely generated $R$-algebra.
We will denote the maximal ideal of a local ring  $R$ by $m_R$, and the quotient field of a domain $R$ by $\mbox{QF}(R)$. 
(We do not require that a local ring be Noetherian). 
Suppose that $R\subset S$ is an inclusion of local rings. We will say that $S$ dominates $R$ if $m_S\cap R=m_R$.
If the local ring $R$ is a domain with $\mbox{QF}(R)=K$ then we will say that $R$ is a local ring of $K$. If $K$ is an algebraic function field over a field $k$ (which we do not assume to be algebraically closed) and a local ring $R$ of $K$ is essentially of finite type over $k$, then we say that $R$ is an algebraic local ring of $k$.  

Suppose that $K\rightarrow K^*$ is a finite field extension, $R$ is a local ring of $K$ and $S$ is a local ring of $K^*$. We will say that $S$ lies over $R$ if $S$ is a localization of the integral closure $T$ of $R$ in $K^*$. If $R$ is a local ring, $\hat R$ will denote the completion of $R$ by its maximal ideal $m_R$.

Suppose that $R$ is a regular local ring. A monoidal transform $R\rightarrow R_1$ of $R$ is a local ring of the form $R[\frac{P}{x}]_m$
where $P$ is a regular prime ideal in $R$ ($R/P$ is a regular local ring) and $m$ is a prime ideal of $R[\frac{P}{x}]$ such that
$m\cap R=m_R$. $R_1$ is called a quadratic transform if $P=m_R$.

\subsection{Valuation Theory}
Suppose that $\nu$ is a valuation on a field $K$. We will denote by $V_{\nu}$ the valuation ring of $\nu$:
$$
V_{\nu}=\{f\in K\mid \nu(f)\ge 0\}.
$$
We will denote the value group of $\nu$ by $\Gamma_{\nu}$. Good treatments of valuation theory are Chapter VI of \cite{ZS2} and \cite{RTM}, which contain references to the original papers.
If $\nu$ is a valuation ring of an algebraic function field over a field $k$, we  insist that $\nu$ vanishes on $k\setminus\{0\}$,
and say that $\nu$ is a $k$-valuation.

If $\nu$ is a valuation of a field $K$ and $R$ is a local ring of $K$ we will say that $\nu$ dominates $R$ if the valuation ring
$V_{\nu}$ dominates $R$. Suppose that $\nu$ dominates $R$. A monoidal transform $R\rightarrow R_1$ is called a monoidal transform along $\nu$ if $\nu$ dominates $R_1$. If $A$ is a subring of $V_{\nu}$, then the center of $\nu$ on $A$ is the prime ideal $m_{\nu}\cap A$. 

Suppose that $K^*/K$ is a finite  extension, $\nu^*$ is a valuation of $K^*$ and $\nu$ is the restriction of $\nu$ to $K$.
We define
$$
e=e(\nu^*/\nu)=|\Gamma_{\nu^*}/\Gamma_{\nu}|
$$
and 
$$
f=f(\nu^*/\nu)=[V_{\nu^*}/m_{\nu^*}:V_{\nu}/m_{\nu}].
$$
$\nu^*/\nu$ is an immediate extension if $e=f=1$.
The defect $\delta(\nu^*/\nu)$ is defined and its basic properties are developed  in  Section 11, Chapter VI \cite{ZS2},
\cite{Ku1}, \cite{Ku2} and
Section 7.1 of \cite{CP}.  If $\nu^*$ is the unique extension of $\nu$ to $K^*$ and $p>0$ is the residue characteristic of $V_{\nu^*}$, then 
$$
[K^*:K]=e(\nu^*/\nu)f(\nu^*/\nu)p^{\delta(\nu^*/\nu)}.
$$

\subsection{Galois theory of local rings} Suppose that $K^*/K$ is a finite Galois extension, $R$ is a normal local ring of $K$ and $S$ is a normal local ring of $K^*$ which lies over $R$. We will denote the Galois group of $K^*/K$ by $G(K^*/K)$. The splitting group $G^s(S/R)$, splitting field $K^s(S/R)=(K^*)^{G^s(S/R)}$  and inertia group $G^i(S/R)$, inertia field $K^i(S/R)=(K^*)^{G^i(S/R)}$   are defined and their basic properties developed in Section 7 of \cite{RTM}. 

\subsection{Galois theory of valuations}
The Galois theory of valuation rings is developed in Section 12 of Chapter VI of \cite{ZS2} and in Section 7 of \cite{RTM}.
Some of the basic results we need are surveyed in Section 7.1 \cite{CP}.
If we take $S=V_{\nu^*}$ and 
$R=V_{\nu}$ where $\nu^*$ is a valuation of $K^*$ and $\nu$ is the restriction of $\nu$ to $K$, then we obtain the splitting group
$G^s(\nu^*/\nu)$, the splitting field $K^s(\nu^*/\nu)$ and the inertia group $G^i(\nu^*/\nu)$, inertia field $K^i(\nu^*/\nu)=(K^*)^{G^i(\nu^*/\nu)}$. 
In Section 12 of Chapter VI of \cite{ZS2}, $G^s(\nu^*/\nu)$ is written as $G_Z$ and called the decomposition group.
$G^i(\nu^*/\nu)$ is written as $G_T$. The ramification group $G_V$ of $\nu^*/\nu$ is  defined in Section 12 of Chapter VI of \cite{ZS2} and is surveyed in Section 7.1 \cite{CP}. We  will denote this group by $G^r(\nu^*/\nu)$.  
\subsection{Semigroups and associated graded rings of a local ring with respect to a valuation}
Suppose that $\nu$ is a valuation of field $K$ which dominates a local ring $R$ of $K$. We will denote the semigroup of values of $\nu$ on $S$ by
$$
S^R(\nu)=\{\nu(f)\mid f\in R\setminus \{0\}\}.
$$
Suppose that $\gamma\in \Gamma_{\nu}$. We define ideals in $R$
$$
\mathcal P_{\gamma}(R)=\{f\in R\mid \nu(f)\ge 0\}
$$
and
$$
\mathcal P_{\gamma}^+(R)=\{f\in R\mid \nu(f)> 0\}
$$
and define (as in \cite{T}) the associated graded ring of $R$ with respect to $\nu$ by
$$
{\rm gr}_{\nu}(R):=\bigoplus_{\gamma\in \Gamma_{\nu}}\mathcal P_{\gamma}(R)/\mathcal P_{\gamma}^+(R)=\bigoplus_{\gamma\in S^R(\nu)}\mathcal P_{\gamma}(R)/\mathcal P_{\gamma}^+(R).
$$

\section{Ramification of associated graded rings of valuations}\label{SecA}

\begin{Lemma}\label{LemmaQ1} Suppose that $R$ is a (not necessarilly Noetherian) local ring with quotient field $K$ which is dominated by a valuation $\nu$ of $K$. Then
$$
{\rm QF}({\rm gr}_{\nu}(R))={\rm QF}({\rm gr}_{\nu}(V_{\nu})).
$$
\end{Lemma}

\begin{proof} For $\gamma\in \Gamma_{\nu}$, $\mathcal P_{\gamma}(V_{\nu})\cap R=\mathcal P_{\gamma}(R)$ and $\mathcal P_{\gamma}^+(V_{\nu})\cap R=\mathcal P_{\gamma}^+(R)$, so we have a natural graded inclusion ${\rm gr}_{\nu}(R)\subset {\rm gr}_{\nu}(V_{\nu})$. 

Suppose $f\in {\rm gr}_{\nu}(V_{\nu})$ is homogeneous, so $f={\rm in}_{\nu}(a)$ for some $a\in V_{\nu}$. Write $a=\frac{b}{c}$ with $b,c\in R$. Since $ca=b$, we have that ${\rm in}_{\nu}(c){\rm in}_{\nu}(a)={\rm in }_{\nu}(b)$, so 
$$
f= \frac{{\rm in }_{\nu}(b)}{{\rm in}_{\nu}(c)}\in {\rm QF}({\rm gr}_{\nu}(R)).
$$
Since every element of ${\rm gr}_{\nu}(V_{\nu})$ is a finite sum of homogeneous elements, we have that ${\rm gr}_{\nu}(V_{\nu})\subset {\rm QF}({\rm gr}_{\nu}(R))$, and the conclusion of the lemma follows.
\end{proof}

\begin{Lemma}\label{LemmaQ2} Suppose that $K\rightarrow K^*$ is a field extension, $\nu^*$ is a valuation of $K^*$ and $\nu|K$ is its restriction to $K$. Suppose that the extension is immediate. Then
$$
{\rm gr}_{\nu}(V_{\nu})={\rm gr}_{\nu^*}(V_{\nu^*}).
$$
\end{Lemma}

\begin{proof} We have that ${\mathcal P}_{\gamma}(V_{\nu^*})\cap K=\mathcal P_{\gamma}(V_{\nu})$ and $\mathcal P_{\gamma}^+(V_{\nu^*})\cap K=\mathcal P_{\gamma}^+(V_{\nu})$ so we have a natural graded inclusion ${\rm gr}_{\nu}(V_{\nu}) \subset {\rm gr}_{\nu^*}(V_{\nu^*})$.
Since $K^*/K$ is immediate, we have $\Gamma_{\nu^*}=\Gamma_{\nu}$ and 
$$
\mathcal P_0(V_{\nu^*})/\mathcal P_0^+(V_{\nu^*})=V_{\nu^*}/m_{\nu^*}=V_{\nu}/m_{\nu}=\mathcal P_0(V_{\nu})/\mathcal P_0^+(V_{\nu}).
$$
Suppose $f\in {\rm gr}_{\nu^*}(V_{\nu^*})$. We will show that $f\in {\rm gr}_{\nu}(V_{\nu})$. We may assume that $f$ is homogeneous, so $f={\rm in}_{\nu^*}(a)$ for some $a\in V_{\nu^*}$. There exists $b\in V_{\nu}$ such that $\nu(b)=\nu^*(a)$. Thus
$$
{\rm in}_{\nu^*}(\frac{a}{b})\in V_{\nu^*}/m_{\nu^*}=V_{\nu}/m_{\nu}
$$
so there exists $c\in V_{\nu}$ such that
$$
{\rm in}_{\nu}(c)={\rm in}_{\nu^*}(\frac{a}{b})=\rm{in}_{\nu^*}(a)-{\rm in}_{\nu}(b).
$$
Thus
$$
f={\rm in}_{\nu^*}(a)={\rm in }_{\nu}(b)+{\rm in}_{\nu}(c)\in{\rm gr}_{\nu}(V_{\nu}).
$$

\end{proof}

\begin{Proposition}\label{PropQ3}  Suppose that $K\rightarrow K^*$ is a finite field extension, $\nu^*$ is a valuation of $K^*$ and $\nu|K$ is its restriction to $K$. Then
$$
[{\rm QF}({\rm gr}_{\nu^*}(V_{\nu^*})):{\rm QF}({\rm gr}_{\nu}(V_{\nu}))]=ef.
$$ 
\end{Proposition}

\begin{proof} Let $w_i\in V_{\nu^*}$ be such that the cosets of the $\nu^*(w_i)$ are a complete set of representatives of $\Gamma_{\nu^*}/\Gamma_{\nu}$, and let $c_j\in V_{\nu^*}$ be such that $\nu^*(c_j)=0$ for all $j$, and the classes of the $c_j$ in $V_{\nu^*}/m_{\nu^*}$ are a basis of $V_{\nu^*}/m_{\nu^*}$ over $V_{\nu}/m_{\nu}$.  Let $\overline w_i={\rm in}_{\nu^*}(w_i)$ and $\overline c_j={\rm in}_{\nu^*}(c_j)$. 

Suppose $f\in {\rm gr}_{\nu^*}(V_{\nu^*})$ is homogeneous and nonzero. Then $f={\rm in}_{\nu^*}(a)$ for some $a\in V_{\nu^*}$. We have that $\nu^*(a)-\nu^*(w_i)\in \Gamma_{\nu}$ for some $w_i$. So 
$$
\nu^*(a)-\nu^*(w_i)=\nu(c)-\nu(d)
$$
for some $c,d,\in V_{\nu}$. We have that
$$
\nu^*(\frac{ad}{w_ic})=0.
$$
Thus there exists $h_j\in V_{\nu}$, with $\nu(h_j)=0$, such that
$$
\sum_j\overline c_j {\rm in}_{\nu}(h_j)={\rm in}_{\nu^*}(\frac{ad}{w_ic}),
$$
so
$$
{\rm in}_{\nu^*}(a){\rm in}_{\nu}(d)=\sum_j{\rm in}_{\nu}(c)\overline w_i\overline c_j{\rm in}_{\nu}(h_j)
$$
and 
$$
f={\rm in}_{\nu^*}(a)\in \sum_{j}{\rm QF}({\rm gr}_{\nu}(V_{\nu}))\overline c_j\overline w_i,
$$
and thus 
$$
\sum_{i,j}{\rm QF}({\rm gr}_{\nu}(V_{\nu}))\overline c_j\overline w_i={\rm QF}({\rm gr}_{\nu^*}(V_{\nu^*})).
$$
Suppose there exist $h_{ij}\in {\rm QF}({\rm gr}_{\nu}(V_{\nu}))$ not all zero such that $\sum_{i,j}h_{ij}\overline c_j\overline w_i=0$. Multiplying be an appropriate nonzero element of ${\rm gr}_{\nu}(V_{\nu})$, we may assume that $h_{ij}\in{\rm gr}_{\nu}(V_{\nu})$ for all $i,j$. Writing $h_{ij}=\sum_{\gamma}h_{ij,\gamma}$ where $h_{ij,\gamma}$ is homogeneous of degree $\gamma$ in ${\rm gr}_{\nu}(V_{\nu})$, we see that $\sum_jh_{ij}\overline c_j=0$ for all $i$. Since $\overline c_j\in\mathcal P_0(V_{\nu^*})/\mathcal P_0^+(V_{\nu^*})$, we have that $\sum_j h_{ij\gamma}\overline c_j=0$ for all $\gamma\in\Gamma_{\nu}$. Suppose some $h_{ij,\gamma}$ is not zero, say $h_{ij_0,\gamma}$. Then
\begin{equation}\label{eqQ10}
\sum_j\left(\frac{h_{ij,\gamma}}{h_{ij_0,\gamma}}\right)\overline c_j=0.
\end{equation}
There exists $a_j\in V_{\nu}$ such that $\nu(a_j)=\gamma$ and ${\rm in}_{\nu}(a_j)=h_{ij,\gamma}$. We have that
$\nu(\frac{a_j}{a_{j_0}})=0$, and 
$$
{\rm in}_{\nu}(\frac{a_j}{a_{j_0}})=\frac{h_{ij,\gamma}}{h_{ij_0,\gamma}}.
$$
Thus (\ref{eqQ10}) is a relation of linear dependence of the $\overline c_j$ over $V_{\nu}/m_{\nu}$, which is impossible. We have that 
$\{\overline c_j\overline w_i\}$ are linearly independent over ${\rm QF}({\rm gr}_{\nu}(V_{\nu}))$, and are thus a basis of ${\rm QF}({\rm gr}_{\nu^*}(V_{\nu^*}))$ over ${\rm QF}({\rm gr}_{\nu}(V_{\nu}))$.
\end{proof}

\begin{Theorem}\label{TheoremQ4} Suppose that $K\rightarrow K^*$ is a finite field extension, $\nu^*$ is a valuation of $K^*$, and $\nu=\nu^*|K$. Suppose that $V_{\nu}/m_{\nu}
=V_{\nu^*}/m_{\nu^*}$ is algebraically closed of characteristic zero. Then
${\rm QF}({\rm gr}_{\nu^*}(V_{\nu^*}))$ is a finite Galois extension of ${\rm QF}({\rm gr}_{\nu}(V_{\nu}))$ with Galois group $\Gamma_{\nu^*}/\Gamma_{\nu}$.
\end{Theorem}

\begin{proof} We will first prove the theorem with the assumptions that $K^*$ is Galois over $K$ and the splitting group
$G^s(\nu^*/\nu)$ of $\nu^*$ over $\nu$ is equal to the Galois group $G(K^*/K)$ of $K^*$ over $K$. With our assumption that $V_{\nu}/m_{\nu}=V_{\nu^*}/m_{\nu^*}$, we have that $G^s(\nu^*/\nu)$ is equal to the inertia group $G^i(\nu^*/\nu)$ of $\nu^*/\nu$ (Theorem 21, page 69 \cite{ZS2}).

Choose $w_i\in V_{\nu^*}$ which are a complete set of representatives of $\Gamma_{\nu^*}/\Gamma_{\nu}$. 
Let $K_1:= {\rm QF}({\rm gr}_{\nu}(V_{\nu}))$ and $K_2:={\rm QF}({\rm gr}_{\nu^*}(V_{\nu^*}))$. By the Corollary to Theorem 25, page 78 \cite{ZS2}, the $w_i$ are a basis of $K^*$ as a $K$-vector space and by the proof of Proposition \ref{PropQ3},
 letting $\overline w_i={\rm in }_{\nu^*}(w_i)$, the $\overline w_i$ are a basis of
$K_2$ as a $K_1$-vector space.

Suppose that $\sigma\in G(K^*/K)$. 
Since $G^s(\nu^*/\nu)=G(K^*/K)$, we have that $\nu^*$ is the unique extension of $\nu$ to $K^*$, so the valuations $\nu^*\sigma$ and $\nu^*$ are equal (formula (3) on page 68 \cite{ZS2}). Thus $\nu^*(\sigma(a))=\nu^*(a)$ for all $a\in K^*$ and so 
$$
\sigma(w_i)=c_iw_i
$$
 with $c_i\in V_{\nu^*}$ such that $\nu^*(c_i)=0$.  

Since the $w_i$ are a basis of $K^*$ over $K$ and they are a complete set of representatives of $\Gamma_{\nu^*}/\Gamma_{\nu}$, there exist uniquely determined index $\lambda(i,j)$, $g_{ij}\in K$ and $h_{ij}\in \sum_{k\ne \lambda(i,j)}Kw_k$ such that
$$
w_iw_j=g_{ij}w_{\lambda(i,j)}+h_{ij}
$$
with 
$$
\nu^*(g_{ij}w_{\lambda(i,j)})=\nu^*(w_iw_j)
$$
and $\nu^*(h_{ij})>\nu^*(w_iw_j)$. We compute
\begin{equation}\label{eqQ21}
\sigma(w_iw_j)=\sigma(g_{ij}w_{\lambda(i,j)}+h_{ij})
=g_{ij}\sigma(w_{\lambda(i,j)})+\sigma(h_{ij})
=g_{ij}c_{\lambda(i,j)}w_{\lambda(i,j)}+\sigma(h_{ij})
\end{equation}
and
\begin{equation}\label{eqQ22}
\sigma(w_iw_j)=\sigma(w_i)\sigma(w_j)=c_ic_jw_iw_j=c_ic_j(g_{ij}w_{\lambda(i,j)}+h_{ij}).
\end{equation}
Comparing (\ref{eqQ21}) and (\ref{eqQ22}), and since 
$$
\nu^*(\sigma(h_{ij})=\nu^*(h_{ij})>\nu^*(g_{ij}w_{\lambda(i,j)},
$$
we obtain 
$$
c_ic_j\equiv c_{\lambda(i,j)}\mbox{ mod }m_{\nu^*}.
$$

Define 
$$
\overline \sigma(\overline w_i)=\overline c_i\overline w_i.
$$
where  $\overline c_i={\rm in}_{\nu^*}(c_i)\in V_{\nu^*}/m_{\nu^*}$.
$\overline \sigma$ extends naturally to a $K_1$-vector space isomorphism of $K_2$. We will show that $\overline\sigma$ preserves the algebra structure on $K_2$, so that $\overline\sigma$ is actually a $K_1$-algebra isomorphism of $K_2$. To check this, we  observe that
 there exist $a_{ij},b_{ij}\in V_{\nu}$ such that $g_{ij}=\frac{a_{ij}}{b_{ij}}$ and so 
$$
\overline w_i\overline w_j=\frac{{\rm in}_{\nu}(a_{ij})}{{\rm in}_{\nu}(b_{ij})}\overline w_{\lambda(i,j)}.
$$
Thus
$$
\overline \sigma(\overline w_i\overline w_j)= \frac{{\rm in}_{\nu}(a_{ij})}{{\rm in }_{\nu}(b_{ij})} \sigma(\overline w_{\lambda(i,j)})
=\frac{{\rm in}_{\nu}(a_{ij})}{{\rm in }_{\nu}(b_{ij})}\overline c_{\lambda(i,j)}\overline w_{\lambda(i,j)}
=\overline c_i\overline w_i\overline c_j\overline w_j
=\overline \sigma(\overline w_i)\overline\sigma(\overline  w_j).
$$
Now from the facts that $G(K^*/K)=G^s(\nu^*/\nu)=G^i(\nu^*/\nu)$, and $G^r(\nu^*/\nu)= \{{\rm id}\}$ (by Theorem 24, page 77 \cite{ZS2}), and that if $\overline c_i=1$ for all $i$ then $\overline\sigma\in G^r(\nu^*/\nu)$ by equation (17), page 75 \cite{ZS2},
we have an injection of groups
$$
G(K^*/K)\rightarrow \mbox{Aut}(K_2/K_1).
$$
From the natural isomorphism 
$$
G(K^*/K)=G^i(\nu^*/\nu)=G^s(\nu^*/\nu)\cong \Gamma_{\nu^*}/\Gamma_{\nu}
$$
of the Corollary of page 77 \cite{ZS2}, and the fact that $[K_2:K_1]=ef=e=|\Gamma_{\nu^*}/\Gamma_{\nu}|$ by Proposition \ref{PropQ3},
we have that $K_2$ is Galois over $K_1$ with Galois group $\Gamma_{\nu^*}/\Gamma_{\nu}$.

We now establish the theorem in  the general case of $K\rightarrow K^*$. Let $K'$ be a Galois closure of $K$. Let $\nu'$ be an extension of
$\nu^*$ to $K'$. Let 
$$
K^s=K^{G^s(\nu'/\nu)}
$$
where
$G^s(\nu'/\nu)\le G(K'/K)$ is the splitting group of $\nu'$ over $\nu$, and let
$$
(K^s)^*=K^{G^s(\nu'/\nu^*)}
$$
where
$G^s(\nu'/\nu^*)\le G(K'/K^*)$ is the splitting group of $\nu'$ over $\nu^*$. The Galois group 
$$
G(K'/K^s)\cong G^s(\nu'/\nu)\cong \Gamma_{\nu'}/\Gamma_{\nu}
$$
 is Abelian ($G^i(\nu'/\nu)=G^s(\nu'/\nu)$ since $V_{\nu}/m_{\nu}$ is algebraically closed) 
so $(K^s)^*$ is Galois over $K^s$, with Galois group $\Gamma_{\nu^*}/\Gamma_{\nu}$.

Let $\overline \nu*=\nu'|(K^s)^*$ and $\overline\nu=\nu'|K^s$. We have that $G^s(\nu^*/\overline\nu)=G((K^s)^*/K^s)$ by Proposition 1.46 \cite{RTM}, since $V_{\nu'}$ is the only local ring of $K'$ lying over $V_{\overline\nu}$. Thus the analysis of the first case holds for $(K^s)^*/K^s$, and we have that ${\rm QF}({\rm gr}_{\overline \nu^*}(V_{\overline \nu^*}))$ is Galois over ${\rm QF}({\rm gr}_{\overline \nu}(V_{\overline \nu}))$ with Galois group $\Gamma_{\nu^*}/\Gamma_{\nu}$. By Lemma \ref{LemmaQ2} or Proposition \ref{PropQ3}, 
${\rm QF}({\rm gr}_{\nu}(V_{ \nu}))={\rm QF}({\rm gr}_{\overline \nu}(V_{\overline \nu}))$ and
${\rm QF}({\rm gr}_{\nu^*}(V_{ \nu^*}))={\rm QF}({\rm gr}_{\overline \nu^*}(V_{\overline \nu^*}))$ so the theorem holds.

\end{proof}

Suppose that $R$ is a (Noetherian) local ring which is dominated by a rank 1 valuation $\nu$. For $f\in \hat R$, we write $\nu(f)=\infty$ if there exists a Cauchy sequence $\{f_n\}$ in $R$ which converges to $f$, and such that
$\lim_{n\rightarrow \infty}\nu(f_n)=\infty$.
We define (Definition 5.2 \cite{CG}) a prime ideal
$$
P(\hat R)_{\infty}=\{f\in \hat R\mid \nu(f)=\infty\}
$$
in $\hat R$. We then have a canonical immediate extension $\hat\nu$ of $\nu$ to ${\rm QF}(\hat R/P(\hat R)_{\infty})$ which dominates $\hat R/P(\hat R)_{\infty}$.

\begin{Lemma}\label{Lemma12} Suppose that $\nu$ is  a rank 1 valuation of a field $K$ and $R$ is a (Noetherian) local ring which is dominated by $\nu$. Let $\hat\nu$ be the canonical extension of $\nu$ to ${\rm QF}(\hat R/P(\hat R)_{\infty})$ which dominates $\hat R/P(\hat R)_{\infty}$. Then the inclusion $R\rightarrow \hat R/P(\hat R)_{\infty}$ induces an isomorphism 
$$
{\rm gr}_{\nu}(R)\cong {\rm gr}_{\hat\nu}(\hat R/P(\hat R)_{\infty}).
$$
\end{Lemma}

\begin{proof}  Suppose $h\in \hat R\setminus P(\hat R)_{\infty}$. There exists a Cauchy sequence $\{f_n\}$ in $R$ such that
$\lim_{n\rightarrow \infty}f_n=h$. Let $m$ be a positive integer such that $m\nu(m_R)>\hat\nu(h)$ (where $\nu(m_R)=\min\{\nu(g)\mid g\in m_R\}$). There exists $n_0$ such that $f_n - h \in m_R^m\hat R$ for $n\ge n_0$. Then ${\rm in }_{\nu}(f_n)={\rm in }_{\hat \nu}(h)$ for $n\ge n_0$.
\end{proof}

\begin{Proposition}\label{Prop42} Suppose that $K\rightarrow K^*$ is a finite field extension, $\nu^*$ is a valuation of $K^*$ and $\nu=\nu^*|K$. Suppose that $R$ and $S$ are respective local rings of $K$ and $K^*$ such that $\nu^*$ dominates $S$ and $S$ lies over $R$. Further suppose that $\mathcal P(\hat S)_{\infty}=(0)$, $\mathcal P(\hat R)_{\infty}=(0)$ and ${\rm gr}_{\nu^*}(S)={\rm gr}_{\nu}(R)$. Then $\hat S=\hat R$.
\end{Proposition}

\begin{proof} For $n\in \ZZ_{\ge 0}$, let
$$
I_n=\{f\in R\mid \nu(f)\ge n\}\mbox{ and }J_n=\{g\in S\mid \nu^*(g)\ge n\}.
$$
Since $\mathcal P(\hat R)_{\infty}=(0)$ and $\mathcal P(\hat S)_{\infty}=(0)$, $\cap_{n\ge 1}(J_n\hat S)=(0)$ and $\cap_{n\ge 1}(I_n\hat R)=(0)$. 
By Chevalley's theorem (Theorem 13, page 270 \cite{ZS2}), the topology on $\hat S$ induced by $\{J_n\hat S\}$ is equivalent to the $m_S\hat S$-adic topology and the topology on $\hat R$ induced by $\{I_n\hat R\}$ is equivalent to the $m_R\hat R$-adic topology. Since ${\rm gr}_{\nu^*}(S)={\rm gr}_{\nu}(R)$,
$\hat S/J_n\hat S\cong S/J_n\cong R/I_n\cong \hat R/I_n\hat R$ for all $n$, so $\hat S\cong \hat R$.
\end{proof}

The assumptions that $\mathcal P(\hat R)_{\infty}=(0)$, $\mathcal P(\hat S)_{\infty}=(0)$  and ${\rm gr}_{\nu^*}(S)={\rm gr}_{\nu}(R)$ are necessary in Proposition \ref{Prop42}, as is shown by the following example.

\begin{Example}\label{Ex4} There exists a finite extension $K\rightarrow K^*$ of algebraic function fields over any ground field $k$ of characteristic $\ne 2$, a discrete rank 1 $k$-valuation $\nu^*$ of $K^*$ with restriction $\nu$ to $K$, and an algebraic regular local ring $R$ of $K$ which is dominated by $\nu$ such that the local ring $S$ obtained by localizing the integral closure of $R$ in $K^*$ at the center of $\nu^*$ is not a regular local ring but the natural inclusion
${\rm gr}_{\nu}(R)\rightarrow {\rm gr}_{\nu^*}(S)$ is an isomorphism.
\end{Example} 

\begin{proof} Let $A=k[x,y]$ and $p(x)=x+a_2x^2+a_3d^3+\cdots$ be a transcendental series in $k[[x]]$. Let $K={\rm QF}(A)=k(x,y)$. Define a valuation $\nu$ on $K$ which dominates $R=A_{(x,y)}$ by $\nu(f)=n$ if $f\in A$ and $\mbox{ord }f(x,p(x))=n$ (computed in $k[[x]]\cong \hat R/(y-p(x))$). We have that $\Gamma_{\nu}\cong\ZZ$ and $V_{\nu}/m_{\nu}\cong k$. (The valuation $\nu$ is the restriction of the order valuation of $k[[x]]$ to $R$ by the natural inclusion $R\rightarrow \hat R/(y-p(x))\cong k[[x]]$. We have that  $P(\hat R)_{\infty}=(y-p(x))$. 

Let $B=k[x,y,z]/(z^2-xy)$. $B$ is the integral closure of $A$ in the quotient field $K^*$ of $B$. Since $\mbox{char } k\ne 2$, in $k[[x,z]]$, we have a factorization
$$
z^2-xp(x)=z^2-x^2-a_2x^3-\cdots =(z-\phi(x))(z-\psi(x))
$$
where $\phi(x)=x+b_2x^2+\cdots$ and $\psi(x)=-x+c_2x^2+\cdots$ for some $b_i,c_i\in k$. Define a valuation $\nu^*$ on $K^*$ by $\nu^*(f)=n$ if $f\in B$ and
$\mbox{ord}_xf(x,p(x),\phi(x))=n$. We have that $\nu^*$ is the restriction of the order valuation of $k[[x]]$ to $S=B_{(x,y,z)}$ by the natural inclusion
$S\rightarrow \hat S/(y-p(x), z-\phi(x))\cong k[[x]]$. We have that $\Gamma_{\nu^*}\cong \ZZ$, $V_{\nu^*}/m_{\nu^*}=k$ and $P(\hat S)_{\infty}=(y-p(x),z-\phi(x))$. We have that $\nu=\nu^*|K$, and the natural inclusion of $R$ into $S$ induces an isomorphism $\hat R/P(\hat R)_{\infty}\cong \hat S/P(\hat S)_{\infty}\cong k[[x]]$. 
Thus the natural inclusion ${\rm gr}_{\nu}(R)\cong {\rm gr}_{\nu^*}(S)$ is an isomorphism by Lemma \ref{Lemma12}. In fact, we have that both graded rings are isomorphic to ${\rm gr} _{xk[[x]]}k[[x]]\cong k[x]$, a $\ZZ$-graded ring with $\deg(x)=1$.
\end{proof}

The following lemma isolated an argument in the proof of Theorem 4.9 \cite{RTM}.

\begin{Lemma}\label{Lemma10} Suppose that $K\rightarrow K^*$ is a finite Galois extension, $\nu^*$ is a valuation of $K^*$ and $\nu=\nu^*|K$. Then there exists a finite set of elements $a_1,\ldots,a_m\in V_{\nu}$ such that if $R$ is a local ring of $K$ which contains $a_1,\ldots,a_m$ and $S$ is the local ring of the integral closure of $R$ in $K^*$ obtained by localizing at the center of $\nu^*$, then 
$$
G^s(S/R)=G^s(\nu^*/\nu).
$$
\end{Lemma}

\begin{proof} Let $\nu_1^*=\nu^*, \nu_2^*,\ldots,\nu_n^*$ be the distinct extensions of $\nu$ to $K^*$, and let $U=\cap_{i=1}^n V_{\nu_i^*}$ be the integral closure of $V_{\nu}$ in $K^*$. Let $H_i=m_{\nu_i^*}\cap U$ for $1\le i\le n$ be the maximal ideals of $U$. There exists $u\in U$ such that $u\in H_1$ and $u\not\in H_i$ for $2\le i\le n$ (Lemma 1.3 \cite{RTM}). Let $u^m+a_1u^{m-1}+\cdots+a_m=0$ with $a_i\in V_{\nu}$ be an equation of integral dependence of $u$ over $V_{\nu}$. 

Suppose that $R$ is a local ring of $K$ which is dominated by $\nu$ and contains $a_1,\ldots,a_m$. Let $T$ be the integral closure of $R$ in $K^*$, and let $P_1=T\cap m_{\nu^*}=T\cap H_1$, so that $S=T_{P_1}$. Then $u\in S\cap H_1=P_1$ and $u\not\in S\cap H_i$ for $2\le i\le n$. Suppose $\sigma\in G^s(S/R)$. Then  $\sigma(P_1)=P_1$ so $\sigma(H_1)=H_1$. Thus $G^s(S/R)\subset G^s(\nu^*/\nu)$. By Proposition 1.50 \cite{RTM}, we must have that
$G^s(\nu^*/\nu)\subset G^s(S/R)$. Thus $G^s(\nu^*/\nu)= G^s(S/R)$.
\end{proof}

\begin{Theorem}\label{Theorem5} Suppose that $K$ is an algebraic function field over an arbitrary field $k$ and $K^*$ is a finite extension of $K$. Suppose that $\nu^*$ is a rank one $k$-valuation of $K^*$.
 Let $\nu$ be the restriction of $\nu^*$ to $K$. 
 Then there exist a finite set of elements $f_1,\ldots,f_n\in V_{\nu}$ such that if 
$R$ is a normal algebraic local ring of $K$ which is dominated by $\nu$ which contains $f_1,\ldots,f_n$ and $S$  is the localization at the center of $\nu^*$ on the integral closure of $R$ in $K^*$, then
 ${\rm gr}_{\nu^*}(S)$ is integral over ${\rm gr}_{\nu}(R)$.
 \end{Theorem}
 
 \begin{proof} First suppose that $K^*$ is Galois over $K$ and $\nu^*$ is the unique extension of $\nu$ to $K^*$. Then $S$ is the integral closure of $R$ in $K^*$. Suppose that $z\in S$. Index $G=G(K^*/K)$ as
 $$
 G(K^*/K)=\{\sigma_1,\ldots,\sigma_r\}.
 $$
 Let 
 $$
 f(x)=\prod_{i=1}^r(x-\sigma_i(z))\in R[x].
 $$
 We expand 
 $$
 f(x)=x^n+S_1x^{n-1}+\cdots+S_n
 $$
 where $n=|G|$ and $S_i$ is the $i$-th elementary symmetric function in $\{\sigma_1(z),\sigma_2(z),\ldots,\sigma_r(z)\}$. Since $\nu^*$ is the unique extension of $\nu$ to $K^*$, $\nu^*(\sigma(z))=\nu^*(z)$ for all $\sigma\in G$, so $\nu^*(S_i)\ge i\nu^*(z)$ for all $i$ and $\nu^*(S_iz^{n-i})\ge \nu^*(z^n)$. We thus have a relation
 $$
 {\rm in }_{\nu^*}(z)^n+\sum_{\nu(S_i)=i\nu^*(z)}{\rm in }_{\nu}(S_i){\rm in}_{\nu^*}(z)^{n-i}=0
 $$
 in $\mathcal P_{n\nu^*(z)}(S)/\mathcal P^+_{n\nu^a(z)}(S)$. Thus ${\rm in }_{\nu}(z)$ is integral over ${\rm gr}_{\nu}(R)$.
 
 Now suppose that $K^*/K$ is separable, but with no other restrictions. Let $K'$ be a Galois closure of $K^*$ over $K$. Let $\nu'$ be an extension of $\nu$ to $K'$ and let $T$ be the localization of the integral closure of $R$ in $K'$ at the center of $\nu'$. Let 
 $K^s=(K^*)^{G^s(\nu'/\nu)}$ be the splitting field of $\nu'$ over $\nu$. 
 By Lemma \ref{Lemma10}, there exist $f_1,\ldots,f_n \in V_{\nu}$ such that if $f_1,\ldots,f_n\in R$,
  then the splitting group $G^s(T/R)=G^s(\nu'/\nu)$. Let $U$ be the localization at the center of $\nu'$ of the integral closure of $R$ in $K^s$. 
 ${\rm gr}_{\nu'}(T)$ is integral over ${\rm gr}_{\nu'}(U)$ by the first part of this proof.
 
 Since ${\rm gr}_{\nu^*}(S)$ is contained in ${\rm gr}_{\nu'}(T)$, we have reduced to establishing that ${\rm gr}_{\nu'}(U)$ is integral over ${\rm gr}_{\nu}(R)$.  By Theorem 1.47 \cite{RTM}, we have that $R/m_R=U/m_U$ and $m_RU=m_U$, so by (10.14) and (10.1) \cite{RES}, $\hat R=\hat S$. Let $\hat \nu$ be an extension of $\hat\nu$ to ${\rm QF}(\hat R)$ which dominates $\hat R$. Then $\Gamma_{\nu}$ is an isolated subgroup of $\Gamma_{\hat\nu}$. We have
 $$
 {\rm gr}_{\nu}(R)\cong \bigoplus_{\gamma\in \Gamma_{\nu}}\mathcal P_{\gamma}(\hat R)/\mathcal P_{\gamma}^+(\hat R)
 \cong \bigoplus_{\gamma\in \Gamma_{\nu'}}\mathcal P_{\gamma}(\hat U)/\mathcal P_{\gamma}^+(\hat U)
 \cong {\rm gr}_{\nu'}(U).
 $$
 
 The remaining case is of a general finite extension $K^*$ over $K$. We have a factorization $K\rightarrow \overline K\rightarrow K^*$ where 
 $\overline K$ is separable over $K$ and $K^*$ is purely inseparable over $\overline K$. Let $\overline \nu$ be the restriction of $\nu^*$ to $\overline K$, and let $A$ be the localization of the integral closure of $R$ in $\overline K$ at the center of $\overline\nu$.  By the first two parts of the proof, we have reduced to showing that ${\rm gr}_{\nu^*}(S)$ is integral over ${\rm gr}_{\overline\nu}(A)$. $S$ is the integral closure of $A$ in $K^*$ since $S$ is the only local ring of $K^*$ lying over $A$. Suppose that $z\in S$. Then there exists an exponent $z^{p^n}$ (where $p$ is the characteristic of $k$) 
 such that $z^{p^n}\in K$. Thus $z^{p^n}\in S\cap K=R$. so we have that ${\rm in}_{\nu^*}(z^{p^n})\in {\rm gr}_{\overline\nu}(A)$.
\end{proof}

\begin{Theorem}\label{Theoreminsep} Suppose that $K$ and $K^*$ are fields, $\nu^*$ is a valuation of $K^*$ with restriction $\nu$ to $K$, $K^*$ is Galois over $K$, $\nu^*$ has rank 1, $\nu^*$ is the unique extension of $\nu$ to $K^*$,
$[\Gamma_{\nu^*}:\Gamma_{\nu}]$ is a power of $p$ where $p$ is the residue characteristic of $V_{\nu*}$ and $V_{\nu^*}/m_{\nu^*}$ is purely inseparable over $[V_{\nu}/m_{\nu}]$
(so $[K^*:K]=p^n$ for some $n$). Suppose $R$ is a normal local ring of $K$ which is dominated by $\nu$ and $S$ is the localization of the integral closure of $R$ in $K^*$ at the center of $\nu^*$. Then ${\rm gr}_{\nu^*}(S)^{p^n}\subset {\rm gr}_{\nu}(R)$.
\end{Theorem}

\begin{proof} With our assumptions, we have that $G(K^*/K)=G^r(\nu^*/\nu)$ (page 68 \cite{ZS2}, Theorem 21, page 69 \cite{ZS2} and Theorem 25, page 76 \cite{ZS2}) and $|G^r(\nu^*/\nu)|=p^n$ for some $n$ (Theorem 24, page 77 \cite{ZS2}).
 Thus $r=[K^*:K]=p^n$ for some positive integer $n$. By page 68 \cite{ZS2},
\begin{equation}\label{eqpi2}
 \nu^*\sigma(x)=\nu^*(x)
 \mbox{ for all $\sigma\in G(K^*/K)$ and $x\in (K^*)^{\times}$} 
 \end{equation}
 and by page 75 \cite{ZS2},
\begin{equation}\label{eqpi1}
\nu^*(\sigma(x)-x)>\nu^*(x)\mbox{ for all $\sigma\in G(K^*/K)$ and $x\in (K^*)^{\times}$}.
\end{equation}

Index $G(K^*/K)$ as $G(K^*/K)=\{\sigma_,\ldots,\sigma_r\}$. For $z\in V_{\nu^*}$, let
$S_i(z)$ be the $i$-th symmetric functions in $\sigma_1(z),\ldots,\sigma_r(z)$. By (\ref{eqpi1}), 
$$
S_i(z)= \binom{p^n}{i}z^i+h_i
$$
with $\nu^*(h_i)>i\nu^*(z)$.
so
\begin{equation}\label{eqpi3}
\nu^*(S_i(z))>i\nu^*(z)\mbox{ for }0<i<p^n.
\end{equation}
Now suppose $h\in {\rm gr}_{\nu^*}(S)$. Since ${\rm gr}_{\nu^*}(S)$ has characteristic $p>0$, we may assume that $h$ is homogeneous to establish that $h^{p^n}\in{\rm gr}_{\nu}(R)$. Then $h={\rm in}_{\nu^*}(z)$ for some $z\in S$.
Let
$$
f(x)=\prod_{i=1}^r(x-\sigma_i(z))=\sum_{i=0}^rS_{r-i}(z)x^i
=x^r+a_1x^{r-1}+\cdots+a_r
$$
with $a_i=S_i(z)\in K\cap S=R$. We have
$$
z^r+a_1z^{r-1}+\cdots+a_r=0.
$$

By (\ref{eqpi3}), ${\rm in}_{\nu^*}(z)^r+{\rm in}_{\nu}(a_r)=0$ in $\mathcal P_{r\nu^*(z)}(S)/\mathcal P^+_{r\nu^*(z)}(S)$. Thus $h^r\in {\rm gr}_{\nu}(R)$.

\end{proof}

\section{An Abhyankar Jung theorem for associated graded rings of valuations}\label{SecB}

\begin{Theorem}\label{Theorem2} Suppose that $K$ is an algebraic function field over an algebraically closed field $k$ of characteristic zero and $K^*$ is a finite extension of $K$. Suppose that $\nu^*$ is a rank one $k$-valuation of $K^*$
whose residue field is $k$. Let $\nu$ be the restriction of $\nu^*$ to $K$, and suppose that $R^*$ is an algebraic local ring of $K$ which is dominated by $\nu$. Then there exists a sequence of monoidal transforms $R^*\rightarrow R_0$ along $\nu$ such that $R_0$ is regular, and if $T$ is the local ring of the center of $\nu^*$ on the integral closure of $R_0$ in $K^*$, then 
\begin{enumerate}
\item[1)] ${\rm gr}_{\nu^*(T)}$ is a free ${\rm gr}_{\nu}(R_0)$-module of finite rank $e=[\Gamma_{\nu^*}/\Gamma_{\nu}]$.
\item[2)] $\Gamma^*/\Gamma$ acts on ${\rm gr}_{\nu^*}(T)$ with ${\rm gr}_{\nu^*}(T)^{\Gamma_{\nu^*}/\Gamma_{\nu}}\cong {\rm gr}_{\nu}(R_0)$.
\end{enumerate} 
\end{Theorem}

With the assumptions of Theorem \ref{Theorem2}, $R_0$ being a regular local ring with the branch divisor of $R_0$ in $K^*$ being a SNC divisor is not enough to obtain the conclusions of Theorem \ref{Theorem2}.  Such an example is given of a rational rank 1 valuation in two dimensional, characteristic zero, algebraic function fields, in Example 9.4 \cite{CV1}. In this example, 
$$
R_0=k[x^2,y^2]_{(x^2,y^2)}\rightarrow T=k[x,y]_{(x,y)}
$$
and ${\rm gr}_{\nu^*}(T)$ is not a finitely generated ${\rm gr}_{\nu}(R_0)$-module. However, by Theorem \ref{Theorem2}, we will find a finitely generated extension after some blowing up above $R_0$.

In Theorem 7.38 \cite{CP}, an example is given of  rational rank 1 valuations in  two dimensional  algebraic function fields, over an algebraically closed field $k$ of positive characteristic, where the branch locus of $R_0$ in $K^*$ is a SNC divisor, $R_0$ and $T$ are regular local rings, the extension is immediate ($e=f=1$) and ${\rm gr}_{\nu^*}(T)$ is not a finitely generated ${\rm gr}_{\nu}(R_0)$-module. In fact, lack of finite generation continues to hold (in this example) after any amount of blowing up above $R_0$.

We summarize some results on ramification. Suppose that $R$ is a normal algebraic  local ring with quotient field $K$ and $K^*$ is a finite separable extension of $K$. 

Let $B(K^*/R)=\sqrt{D(K^*/R)}$ where $D(K^*/R)$ is the discriminant ideal of $R\rightarrow K^*$ (page 31, \cite{RTM}). For $\mathfrak p$ a prime ideal of $R$, $R_{\mathfrak p}$ is unramified in $K^*$ if and only if $B(K^*/R)\not\subset \mathfrak p$ (Theorems 1.44 and 1.44 A \cite{RTM}).

Suppose that $S$ is the localization of the integral closure of $R$ in $K^*$ at a maximal ideal. 
If $R$ has residue characteristic zero, then the ramification index of $S$ over $R$ is
$$
r(S:R)=\frac{[{\rm QF}(\hat S):{\rm QF}(\hat R)]}{[S/m_S:R/m_R]}
$$
(Definition 1 \cite{LU} or Definition 1 \cite{RAF}).
Further,
$r(S:R)=1$ if and only if $R\rightarrow S$ is unramified (Lemma 4 \cite{LU}).

We now summarize some results on toric rings from \cite{BG}.
Suppose that $M$ is a finitely generated submonoid (subsemigroup) of $\ZZ^n$ for some $n\ge 0$. Let
$$
\tilde M_{\ZZ^n}=\{v\in \ZZ^n\mid mv\in M\mbox{ for some }m\in \ZZ_{>0}\}.
$$
We have that $\tilde M_{\ZZ^n}=(\RR_{\ge 0} M)\cap \ZZ^n$ (Proposition 2.2 \cite{BG}).

\begin{Proposition}\label{SG1}(Proposition 2.43 \cite{BG}) Suppose $v_1,\ldots,v_n\in \ZZ^n_{\ge 0}$ are linearly independent. Let
$$
{\rm par}(v_1,\ldots,v_n)=\{q_1v_1+\cdots+q_nv_n\mid 0\le q_i<1\mbox{ for }i=1,\ldots,n\}.
$$
Let $M$ be the submonoid of $\ZZ^n$ generated by $v_1,\ldots,v_n$. Then
\begin{enumerate}
\item[a)] $\Lambda =\ZZ^n\cap {\rm par}(v_1,\ldots,v_n)$ is a system of generators of the $M$-module $\tilde M_{\ZZ^n}$.
\item[b)] $(a+M)\cap(b+M)=\emptyset$ for $a,b\in \Lambda$ with $a\ne b$.
\item[c)] $|\Lambda|=[\QQ U\cap \ZZ^n:U]$ where $U$ is the sublattice of $\ZZ^n$ generated by $M$.
\end{enumerate}
\end{Proposition}

\begin{Lemma}\label{SG2} (Lemma 4.40 \cite{BG}) Suppose that $k$ is a ring, $M$ is a finitely generated submonoid of $\ZZ^n$ and $k[z_1,\ldots,z_n]$ is a polynomial ring over $k$. Then $k[z^v\mid v\in \tilde M_{\ZZ^n}]$ is the integral closure of $k[v\mid v\in M]$ in $k[z_1,\ldots, z_n]$. 
\end{Lemma}

The prime ideals $P(\hat R)_{\infty}$ are defined before Lemma \ref{Lemma12}.

\begin{Theorem}\label{Theorem1} Let $k$ be a field of characteristic zero, $K$ an algebraic function field over $k$, $K^*$ a finite algebraic extension of $K$, $\nu^*$ a $k$-valuation of $K^*$, $\nu=\nu^*|K$ such that rank $\nu=1$ and  rat rank $\nu=s$. Let $n={\rm trdeg}_kK-{\rm trdeg}_kV_{\nu}/m_{\nu}$, $e=e(\nu^*/\nu)$, $f=f(\nu^*/\nu)$. Let $g_1,\ldots,g_f$ be a basis of $V_{\nu^*}/m_{\nu^*}$ over $V_{\nu}/m_{\nu}$. 

Suppose that $S^*$ is an algebraic local ring of $K^*$ which is dominated by $\nu^*$ and 
$R^*$ is an algebraic local ring of $K$ which is dominated by $\nu$ and $S^*$. Then there exists a commutative diagram
$$
\begin{array}{llll}
R_0&\rightarrow &S&\subset V_{\nu^*}\\
\uparrow&&\uparrow&\\
R^*&\rightarrow&S^*
\end{array}
$$
where $S^*\rightarrow S$ and $R^*\rightarrow R_0$ are sequences of monoidal transforms along $\nu^*$ such that $R_0$ has regular parameters $(x_1,\ldots,x_n)$ and $S$ has regular parameters $(y_1,\ldots,y_n)$ such that there are units $\delta_1,\ldots,\delta_s\in S$ and a $s\times s$ matrix $A=(a_{ij})$ of natural numbers (elements of $\ZZ_{\ge 0}$) such that ${\rm Det}(A)\ne 0$,
\begin{equation}\label{eqA2}
\begin{array}{lll}
x_1&=& \delta_1 y_1^{a_{11}}\cdots y_s^{a_{1s}}\\
&\vdots&\\
x_2&=& \delta_s y_1^{a_{s1}}\cdots y_s^{a_{ss}}\\
x_{s+1}&=& y_{s+1}\\
&\vdots&\\
x_n&=& y_n
\end{array}
\end{equation}
and $\{\nu(x_1),\ldots,\nu(x_s)\}$, $\{\nu^*(y_1),\ldots,\nu^*(y_s)\}$ are rational bases of 
$\Gamma_{\nu}\otimes \QQ$ and $\Gamma_{\nu^*}\otimes\QQ$ respectively. Furthermore, there exists $\lambda$ with $\lambda<n-s$ such that
\begin{equation}\label{eqA3}
p(\hat R_0)_{\infty}=(h_1,\ldots,h_{n-\lambda})
\end{equation}
with 
\begin{equation}\label{eqA4}
h_i\equiv x_{s+i}\mbox{ mod }m_{\hat R_0}^2
\end{equation}
for $1\le i\le n-\lambda$ and
\begin{equation}\label{eqA5}
p(\hat S)_{\infty}=p(\hat R_0)_{\infty}\hat S
\end{equation}
are regular primes ($\hat R_0/P(\hat R_0)_{\infty}$ and $\hat S/P(\hat S)_{\infty}$ are regular local rings).

There exists a (unique) normal algebraic local ring $R$ of $K$ which is dominated by $\nu$ such that $S$ is the localization of the integral closure of $R$ in $K^*$ at the center of $\nu^*$.
We have that
\begin{equation}\label{eqA6}
[S/m_S:R/m_R]=f, |{\rm Det}(A)|=e, [{\rm QF}(\hat S):{\rm QF}(\hat R)]=ef
\end{equation}
and \begin{equation}\label{eqA7}
\{g_1,\ldots,g_f\}\mbox{ is a basis of $S/m_S$ over $R/m_R=R_0/m_{R_0}$}.
\end{equation}
We further have that 
\begin{equation}\label{eqA11}
\ZZ^s/A^t\ZZ^s  \cong \Gamma_{\nu^*}/\Gamma_{\nu}
\end{equation}
  by the map $(b_1,\ldots,b_s)\mapsto b_1\nu^*(y_1)+\cdots+b_s\nu^*(y_s)$.

Let $T$ be the localization of the integral closure of $R_0$ in $K^*$ at the center of $\nu^*$.
We have that
\begin{equation}\label{eqA8}
[T/m_T:R_0/m_{R_0}]=f
\end{equation}
and $\{g_1,\ldots,g_f\}$ is a basis of $T/m_T$ over $R_0/m_{R_0}$.
Further,
\begin{equation}\label{eqA12}
p(\hat T)_{\infty}=p(\hat R_0)_{\infty}\hat T,
\end{equation}
and
\begin{equation}\label{eqA9}
[{\rm QF}(\hat T):{\rm QF}(\hat R_0)]=ef.
\end{equation}
The branch ideal $B(K^*/R_0)$ of $R_0\rightarrow K^*$  contains $\prod_{i=1}^sx_i$.

\end{Theorem}

\begin{proof}
Let $R'$ be the algebraic local ring of $K$ constructed in  Theorem 6.1 \cite{CP}, and let $\tilde R$ be the algebraic local ring of $K$ constructed in Theorem 6.4 \cite{CG}, with $\lambda=\lambda_{V_{\nu}}$ as defined before Theorem 6.4 \cite{CG}.

We begin (as in the proof of Theorem 5.1 \cite{C}) by constructing a commutative diagram
$$
\begin{array}{lll}
R_1&\rightarrow &S_1\\
\uparrow&&\uparrow\\
R^*&\rightarrow &S^*
\end{array}
$$
where $R_1$ and $S_1$ are regular local rings which are dominated by $\nu^*$ such that $S_1$ dominates $R_1$, the vertical arrows are products of monoidal transforms, $V_{\nu^*}/m_{\nu^*}$ is algebraic over $S_1/m_{S_1}$ and $R_1/m_{R_1}$, there is a regular system of parameters $x_1(1),\ldots,x_n(1)$ in $R_1$ and a regular system of parameters $y_1(1),\ldots,y_n(1)$ in $S_1$, an $s\times s$ matrix $A(1)=(a_{ij}(1))$ of nonnegative integers with ${\rm Det}(A(1))\ne 0$  and units $\delta_i(1)\in S_1$ for $1\le i\le s$ such that $x_i(1)=\delta_i(1)\prod_{j=1}^sy_j(1)^{a_{ij}(1)}$ for $1\le i\le s$, and $R_1$ dominates both $R'$ and $\tilde R$.

Let $T_1$ be the integral closure of $R_1$ in $K^*$ As $R_1$ is regular, $B(K^*/R_1)$ is a principal ideal by the Purity of the Branch Locus. Let $f$ be such that   $B(K^*/R_1)=(f)$.  By   Theorems 4.7, 4.8 and 4.10 of \cite{C},
 there exists a commutative diagram
$$
\begin{array}{lll}
R_2&\rightarrow &S_2\\
\uparrow&&\uparrow\\
R_1&\rightarrow &S_1
\end{array}
$$
where $R_2$ and $S_2$ are regular local rings which are dominated by $\nu^*$ such that $S_2$ dominates $R_2$, 
the vertical arrows are products of monoidal transforms,  there is a regular system of parameters $x_1(2),\ldots,x_n(2)$ in $R_2$ and a regular system of parameters $y_1(2),\ldots,y_n(2)$ in $S_2$, an $s\times s$ matrix $A(2)=(a_{ij}(2))$ of nonnegative integers with ${\rm Det}(A(2))\ne 0$  and units $\delta_i(2)\in S_2$ for $1\le i\le s$ such that $x_i(2)=\delta_i(2)\prod_{j=1}^sy_j(2)^{a_{ij}(2)}$ for $1\le i\le s$, and 
there exists  a  unit $\epsilon(2)\in R_2$ and $d_i(2)\in \ZZ_{\ge 0}$ for $1\le i\le s$ such that 
$$
f=\epsilon(2)\prod_{i=1}^sx_i(2)^{d_{i}(2)}.
$$
Now proceeding as in the proof of Theorem 6.5 \cite{CG}, we next construct a diagram
$$
\begin{array}{lll}
R_3&\rightarrow &S_3\\
\uparrow&&\uparrow\\
R_2&\rightarrow &S_2
\end{array}
$$
such that the conclusions of Theorem 5.1 \cite{C} hold (we obtain an expression (\ref{eqA2}) with ${\rm Det}(A)\ne 0$), and then  construct a diagram 
$$
\begin{array}{lll}
R_0&\rightarrow &S\\
\uparrow&&\uparrow\\
R_2&\rightarrow &S_2
\end{array}
$$
such that the conclusions of Theorem 6.5 \cite{CG} hold. This is possible since $R_2$ dominates $\tilde R$.
Since $R$ dominates $R'$, the conclusions of Theorem 6.1 \cite{CP} hold. The isomorphism (\ref{eqA11}) is explained in the proof of Theorem 4.10 \cite{CP}).

The argument of the proof of Theorem 6.1 \cite{CP} shows that $[T/m_T:R_0/m_{R_0}]=f$ and $\{g_1,\ldots,g_f\}$ is a basis of $T/m_T$ over $R_0/m_{R_0}$ (since $R_0$ dominates $R'$). 

All transformations in the construction above $R_1\rightarrow S_1$ are CTUTS (page 29 \cite{C}) in $m\ge s$ variables (page 49 \cite{C}) so the branch ideal of $R_0\rightarrow K^*$ contains $\prod_{i=1}^sx_i$. The argument of the proof of Theorem 6.5 \cite{CG} shows that (\ref{eqA12}) holds (since $R_0$ dominates $\tilde R$).

It remains to show that  $[{\rm QF}(\hat T):{\rm QF}(\hat R_0)]=ef$. The ring $R'$ from Theorem 6.1 \cite{CP} was constructed in that proof so that it has the following property. There is a Galois closure $K'$ of $K^*/K$ with extension $\nu'$ of $\nu^*$ to $K'$ such that if $T'$ is the localization of the integral closure of $T$ in $K'$  at the center of $\nu'$, then we have equality of splitting groups
\begin{equation}\label{eqA1}
G^s(T'/T)=G^s(\nu'/\nu^*)\mbox{ and }G^s(T'/R_0)=G^s(\nu'/\nu).
\end{equation}
By the Corollary to Theorem 25, page 78 \cite{ZS2}, we have that 
$$
[K':K^s(\nu'/\nu)]=e(\nu'/\nu)f(\nu'/\nu)\mbox{ and }
[K':K^s(\nu'/\nu^*)]=e(\nu'/\nu^*)f(\nu'/\nu^*).
$$
Let $T_1=K^s(\nu'/\nu)\cap T'$ and $T_2=K^s(\nu'/\nu^*)\cap T'$.
$R_0\rightarrow T_1$ and $T\rightarrow T_2$ are unramified with
$T_1/m_{T_1}=R_0/m_{R_0}$ and $T_2/m_{T_2}=T/m_T$ by equation (\ref{eqA1}) and Theorem 1.47 \cite{RTM}.
Thus $\hat T_1=\hat R_0$ and $\hat T_2=\hat T$ by (10.14) and (10.1) \cite{RES}. We have that
$$
[K':K^s(\nu'/\nu)]=[{\rm QF}(\hat T_1):{\rm QF}(\hat R_0)]\mbox{ and }
[K':K^s(\nu'/\nu^*)]=[{\rm QF}(\hat T_2):{\rm QF}(\hat T)]
$$
by II of Proposition 1 (page 498) \cite{LU}, since there is a unique local ring in $K'$ lying above $T_1$
and a unique local ring in $K'$ lying above $T_2$.
Finally, we have that
$$
[{\rm QF}(\hat T):{\rm QF}(\hat R_0)]=\frac{[{\rm QF}(\hat T'):{\rm QF}(\hat R_0)]}{[{\rm QF}(\hat T'):{\rm QF}(\hat T)]}=e(\nu^*/\nu)f(\nu^*/\nu)
$$
since $e$ and $f$ are multiplicative.
\end{proof}

We will now give the proof of Theorem \ref{Theorem2}. Taking $S^*$ to be the localization of the integral closure of $R^*$ in $K^*$ at the center of $\nu^*$, we may assume that the
 assumptions and conclusions of Theorem \ref{Theorem1} (with the additional assumptions that $k$ is algebraically closed and $V_{\nu^*}/m_{\nu^*}=k$, so that $f=1$). 

We have  a canonical  immediate extension of $\nu^*$ to the the quotient field of $\hat S/p(\hat S)_{\infty}$ which dominates $\hat S/p(\hat S)_{\infty}$, and identifying this extension with $\nu^*$, we can define an extension $\hat \nu^*$ of $\nu$ to the quotient field of $\hat S$ which dominates $\hat S$ with value group $\ZZ^{n-\lambda}\times\Gamma_{\nu^*}$
in the lex order. The value $\hat\nu^*(f)$ for $0\ne f\in \hat S$ is defined as follows.
Let $m_1$ be such that $h_1^{m_1}$ is the largest power of $h_1$ which divides $f$ in $\hat S$. Let $f_1$ be the residue of $\frac{f}{h_1^{m_1}}$ in the regular local ring
$\hat S/h_1\hat S$. Let $m_2$ be the largest power of (the residue of) $h_2$ in $\hat S/h_1\hat S$ such that $h_2^{m_2}$ divides $f_1$ in $\hat S/h_1\hat S$. Continue this way constructing $f_1,\ldots,f_{n-\lambda}$ and $m_1,\ldots,m_{n-\lambda}$, where $f_{n-\lambda}$ is the residue of $f_{n-\lambda+1}$ in $\hat S/p(\hat S)_{\infty}$.
Define 
\begin{equation}\label{eqA31}
\hat\nu^*(f)=(m_1,\ldots,m_{n-\lambda},\nu^*(f_{n-\lambda}))\in \ZZ^{n-\lambda}\times \Gamma_{\nu^*}\mbox{ in the lex order}.
\end{equation}
There exists $z_1,\ldots,z_s\in \hat S$ and units $\epsilon_i\in \hat S$ such that $z_i=\epsilon_iy_i$ and
$x_i=z_1^{a_{i1}}\cdots z_s^{a_{ss}}$ for $1\le i\le s$. Define $z_i=y_i$ for $s+1\le i\le n$.

The following lemma can also be established by the argument before (\ref{eqA40}).

\begin{Lemma}\label{LemmaA23} There exist Laurent monomials  $N_1,\ldots,N_{\ell}$ in $x_1,\ldots,x_n$ such that $\hat R=k[[N_1,\ldots,N_{\ell}]]$.
\end{Lemma}

\begin{proof} As shown in Theorem 4.2 \cite{CP}, $R$ is the localization of the integral closure of $R_0[f_1,\ldots,f_s]$ in $K$ at the center of $\nu$, where $f_j=\prod_{j=1}^sx_j^{b_{ij}}$, with $(b_{ij})$ the adjoint matrix of $A=(a_{ij})$. By Lemma \ref{SG2}, we can extend $f_1,\ldots,f_s$ to a set of Laurent monomials $f_1,\ldots,f_{\ell}$ ($\ell\ge s$) such that $k[x_1,\ldots,x_n,f_1,\ldots,f_{\ell}]$ is the integral closure of $k[x_1,\ldots,x_n,f_1,\ldots,f_{s}]$ in $k(x_1,\ldots,x_n)$. Let $$
C=k[x_1,\ldots,x_n,f_1,\ldots,f_{\ell}]_{(x_1,\ldots,x_n,f_1,\ldots,f_{\ell})}\mbox{ and }D=R_0[x_1,\ldots,x_n,f_1,\ldots,f_{\ell}]_{(x_1,\ldots,x_n,f_1,\ldots,f_{\ell})}.
$$
Since $\widehat{k[x_1,\ldots,x_n]_{(x_1,\ldots,x_n)}}=\hat R_0$, $\hat C=\hat D$. Since $C$ and $D$ are excellent, by (v) of IV.7.8.3 \cite{EGA} $C$ normal implies $\hat C=\hat D$ is normal and so $D$ is normal. Thus $D=R$ by Zariski's main theorem (10.8) \cite{RES}.

\end{proof}

The next part of the proof (through the paragraph before (\ref{eqA30})) is as in the proofs of Theorems 2 and 3 of \cite{RAF}.

Let $K'$ be a Galois closure of $K^*/K$. Then $B(K'/R_0)=B(K^*/R_0)$ (by Proposition 1 \cite{RAF}), which contains
$x_1\cdots x_s$, by the last statement of Theorem \ref{Theorem1}.

Let $\nu'$ be an extension of $\nu^*$ to $K'$ and let $R'$ be the center of $\nu'$ on the integral closure of $R_0$ in $K'$. Let $K_s$ be the splitting field of $R'$ over $R_0$ and $K_s^*$ be the splitting field of $R'$ over $T$. Let
$\overline K'$ be the Galois closure of $K_s^*$ over $K_s$ in $K'$. Let $R_s$ be the localization of the integral closure of $R_0$ in $K_s$ at the center of $\nu'$ and let $T_s$ be the localization of the integral closure of the integral closure of $T$ in $K_s'$ at the center of $\nu'$. Now $x_1\cdots x_s\in B(K'/R_s)$ and so $x_1\cdots x_s\in B(\overline K'/R_s)$ (since $B(K'/R_s)\subset B(\overline K'/R_s)$).

We have that $\hat R_0\cong \hat R_s$ and $\hat T\cong \hat T_s$ (Theorem 1.47 \cite{RTM} and (10.14), (10.1) \cite{RES}).
For $1\le j\le s$, set $n_j=r(A_j,(R_s)_{(x_j)})$, where $A_j$ is any local ring of the integral closure of $R_s$ in the Galois extension $\overline K'/K_s$ which dominates $(R_s)_{(x_j)}$. 

Let $K_1=K_s(x_1^{\frac{1}{n_1}},\ldots,x_s^{\frac{1}{n_s}})$ and $K_2=\overline K'(x_1^{\frac{1}{n_1}},\ldots,x_s^{\frac{1}{n_s}})$. 
Let $\overline\nu'$ be an extension of $\nu'$ to $K_2$. Let $R_1$ be the localization of the integral closure of
$R_s$ in $K_1$ at the center of $\overline\nu'$. By Lemma 6 \cite{RAF}, $R_1$ is unramified in $K_2$ in codimension 1. Now $R_1$ is a regular local ring with $m_{\hat R_1}=(x_1^{\frac{1}{n_1}},\ldots,x_s^{\frac {1}{n_s}},x_{s+1},\ldots,x_n)$ (by Proposition 1 \cite{LU} and Lemma 5 \cite{RAF}). Thus $R_1\rightarrow K_2$ is unramified by the purity of the branch locus (Theorem 1 \cite{RAF}). We have that $\hat R_1=k[[x_1^{\frac{1}{n_1}},\ldots,x_s^{\frac{ 1}{n_s}},x_{s+1},\ldots,x_n]]$. Let $R_2$ be the localization of the integral closure of $R_1$ in $K_2$ at the center of $\overline\nu'$. Then $\hat R_1\cong\hat R_2$ since $R_1$ is unramified in $K_2$.   Let
$$
F_0={\rm QF}(\hat R_0)=k((x_1,\ldots,x_n))\mbox{ and }E_0={\rm QF}(\hat T).
$$
Let
$$
L_0={\rm QF}(\hat R_2)=k((x_1^{\frac{1}{n_1}},\ldots,x_s^{\frac{ 1}{n_s}},x_{s+1},\ldots,x_n)).
$$
$L_0$ is a Galois extension of $F_0$, with abelian Galois group $G(L_0/F_0)\cong \bigoplus_{i=1}^s\ZZ_{n_i}$, acting diagonally on  $x_i^{\frac{1}{n_i}}$ by multiplication by $n_i$-th roots of unity. 

Thus 
$E_0$ is Galois over $F_0$ and 
$E_0$ has a basis over $F_0$ of the form
\begin{equation}\label{eqA30}
\{M_j=\prod_{j=1}^sx_i^{\frac{b_{ij}}{n_j}}\}\mbox{ for }1\le j\le e
\end{equation}
with all $b_{ij}\in \ZZ_{\ge 0}$. We have that $[E_0:F_0]=e$ by (\ref{eqA9}). Recall that $R$ is defined before equation (\ref{eqA6}) in 
Theorem \ref{Theorem1}.
Let
$$
F_1={\rm QF}(\hat R)\mbox{ and }E_1={\rm QF}(\hat S)=k((z_1,\ldots,z_n)).
$$
We have that $E_1=F_1(K^*)$ and $E_0=F_0(K^*)$ (by Proposition 1 \cite{LU}). Thus $M_1,\ldots,M_e$ generate $E_1$ over $F_1$. Since $e=[E_1:F_1]$ (by (\ref{eqA6})), $\{M_1,\ldots,M_e\}$ is a basis of $E_1$ over $F_1$. We have that $z_i$ is a Laurent monomial in $x_1^{\frac{1}{e}},\ldots,x_s^{\frac{1}{e}}$ for $1\le i\le s$, so by Lemma \ref{LemmaA23},  for each $1\le i\le s$, we have an expression
\begin{equation}\label{eq30}
z_i=N_iM_{\sigma(i)}
\end{equation}
where $N_i$ is a Laurant monomial in $x_1,\ldots,x_n$ and $\sigma(i)\in \{1,\ldots,e\}$.
In particular, $z_1,\ldots,z_n\in E_0$.

Let $a_i=(a_{i1},\ldots,a_{is},0,\ldots,0)$ for $1\le i\le s$ (where $A=(a_{ij})$)and let $e_j$ be the row vector of length $n$ with a 1 in the $j$-th place and zeros everywhere else. Let $M$ be the submonoid of $\ZZ^n$ generated by 
$a_1,\ldots, a_s, a_{s+1}=e_{s+1},\ldots,a_n=e_n$. 
$k[z^v\mid v\in \tilde M_{\ZZ^n}]$ is the integral closure of $k[z^v\mid v\in M]=k[x_1,\ldots,x_n]$ in $k[z_1,\ldots,z_n]$ by Lemma \ref{SG2}.

$\hat T$ is integrally closed since $T$ is. Let $A=k[z^{\nu}\mid \nu \in \tilde M_{\ZZ^n}]\subset k(z_1,\ldots,z_n)\subset E_0$. All elements of $A$ are integral over $k[x_1,\ldots,x_n]$ and hence over $\hat R_0=k[[x_1,\ldots,x_n]]$. Thus $A\subset \hat T$.

Let $m_A$ be the maximal ideal of $A$ generated by the monomials of $A$ of positive degree (in $z_1,\ldots,z_n$). 
$\hat T$ is the invariant ring of the action of the Galois group $G(L_0/E_0)$ on $k[[x_1^{\frac{1}{n_1}},\ldots,x_s^{\frac{1}{n_s}},x_{s+1},\ldots,x_n]]$, so the maximal ideal of $\hat T$ is generated by monomials of positive degree in $x_1^{\frac{1}{n_1}},\ldots,x_s^{\frac{1}{n_s}},x_{s+1},\ldots, x_n$.
Thus $m_A\subset M_{\hat T}$ by (\ref{eq30}). In particular, the completion $\hat A$ at the maximal ideal $m_A$ is contained in $\hat T$. $\hat A$ is integrally closed since $A$ is. 

With the notation of Proposition \ref{SG1}, let $\Lambda = \ZZ^n\cap {\rm par}(a_1,\ldots,a_n)$. We have that 
$$
|\Lambda|=[\ZZ^s:A\ZZ^s]=|{\rm Det}(A)|=e
$$
by c) of Proposition \ref{SG1} and equation (\ref{eqA6}). Index the elements of $\{z^v\mid v\in \Lambda\}$ as
$w_1=1,w_2,\ldots,w_e$. We have that $A=\bigoplus_{i=1}^ek[x_1,\ldots,x_n]w_i$ as a $k[x_1,\ldots,x_n]$-module by a) and b) of Proposition \ref{SG1}. Completing with respect to the maximal ideal $(x_1,\ldots,x_n)$ of $k[x_1,\ldots,x_n]$, we have that
$$
A\otimes_{k[x_1,\ldots,x_n]}\hat R_0\cong \bigoplus_{i=1}^ew_i\hat R_0.
$$
We have that $\sqrt{(x_1,\ldots,x_n)A}=m_A$, so that completion of $A$ with respect to $m_A$ is
$\hat A\cong A\otimes_{k[x_1,\ldots,x_n]}\hat R_0$ (Theorem 16, page 277 \cite{ZS2}). Now $w_1,\ldots,w_e$ is a basis of ${\rm QF}(\hat A)$ over ${\rm QF}(\hat R_0)=F_0$, so $[{\rm QF}(\hat A):F_0]=e$.
As ${\rm QF}(\hat A)\subset E_0$ and $[E_0:F_0]=e$, we have that ${\rm QF}(\hat A)=E_0$. As $\hat A$ is integrally closed, we have that $\hat A=\hat T$. Thus
$$
\hat T=\bigoplus_{i=1}^ew_i\hat R_0.
$$

Recall the extension $\hat\nu^*$ of $\nu^*$ to the quotient field $E_1$ of $\hat S$ defined by (\ref{eqA31}). Notice that $\Gamma_{\hat{\nu}^*}=\ZZ^{n-\lambda}\times \Gamma_{\nu^*}$ contains $\Gamma_{\nu^*}$ as an isolated subgroup (page 40 \cite{ZS2}).

\begin{Lemma}\label{LemmaA22} We have that $\hat\nu^*(w_i)\in \Gamma_{\nu^*}$ for $1\le i\le e$ and 
$\hat\nu^*(w_i)$ for $1\le i\le e$ are a complete set of representatives of the cosets of $\Gamma_{\nu}$ in $\Gamma_{\nu^*}$. 
\end{Lemma}

\begin{proof} Each $w_i$ is a monomial in $z_1,\ldots,z_s$ by construction, so $\hat\nu^*(w_i)\in \Gamma_{\nu^*}$ for $1\le i\le e$ (since $\hat\nu^*(z_i)=\hat\nu^*(y_i)$ for all $i$). Write
\begin{equation}\label{eqA22}
w_i=z_1^{c_{i1}}\cdots z_s^{c_{is}}
\end{equation}
for $1\le i\le e$. Suppose $\hat\nu^*(w_i)-\hat\nu^*(w_j)\in \Gamma_{\nu}$ for some $i\ne j$. Then
$$
\hat\nu^*(\frac{w_i}{w_j})\in \Gamma_{\nu}
$$
implies 
$$
\frac{w_i}{w_j}=x_1^{f_1}\cdots x_s^{f_s}
$$
for some $f_1,\ldots,f_s\in \ZZ$
by (\ref{eqA11}). Thus 
$$
\left(\prod_{f_l>0}x_l^{f_l}\right)w_j=\left(\prod_{f_l<0}x_l^{-f_l}\right)w_i.
$$ 
But this is impossible since the $w_i$ are a free basis of $\hat T$ over $\hat R_0$. Thus the classes of $\hat \nu^*(w_i)$ in $\Gamma_{\nu^*}/\Gamma_{\nu}$ are all distinct. 

The lemma now follows since $e=|\Gamma_{\nu^*}/\Gamma_{\nu}|$.
\end{proof}

\begin{Lemma}\label{LemmaA21} There are natural isomorphisms of graded rings
$$
{\rm gr}_{\nu}(R_0)\cong \bigoplus_{\gamma\in \Gamma_{\nu}}\mathcal P_{\gamma}(\hat R_0)/\mathcal P_{\gamma}^+(\hat R_0)
$$
and
$$
{\rm gr}_{\nu^*}(T)\cong \bigoplus_{\gamma\in \Gamma_{\nu^*}}\mathcal P_{\gamma}(\hat T)/\mathcal P_{\gamma}^+(\hat T).
$$
\end{Lemma}

\begin{proof} Suppose that $\gamma\in \Gamma_{\nu}$. Then 
$\mathcal P_{\gamma}(\hat R_0)\cap R_0=\mathcal P_{\gamma}(R_0)$ and $\mathcal P^+_{\gamma}(\hat R_0)\cap R_0=\mathcal P^+_{\gamma}(R_0)$. Thus the natural map
\begin{equation}\label{eqA20}
\mathcal P_{\gamma}(R_0)/\mathcal P_{\gamma}^+(R_0)
\rightarrow \mathcal P_{\gamma}(\hat R_0)/\mathcal P_{\gamma}^+(\hat R_0)
\end{equation}
is well defined and injective.
Suppose $F\in \mathcal P_{\gamma}(\hat R_0)$. Then there exists $m\in \ZZ_{>0}$ such that $\hat{\nu}^*(F)=\gamma<m\nu(m_{R_0})$ (since $\gamma\in\Gamma_{\nu}$) and there exist $F'\in R_0$, $h\in m_{\hat R_0}^m$ such that $F'=F+h$. Then
$\hat\nu^*(F')=\hat\nu^*(F)=\gamma$ and 
$$
F'\equiv F\mbox{ mod }\mathcal P_{\gamma}^+(\hat R_0).
$$
Thus (\ref{eqA20}) is an isomorphism for all $\gamma\in \Gamma_{\nu}$. The same argument applied to $T$ establishes the lemma.
\end{proof}

Now we return to the proof of Theorem \ref{Theorem2}.  Suppose $f\in \hat T$ and $\hat\nu^*(f)\in \Gamma_{\nu^*}$. Write
$$
f=\sum_{i=1}^e f_iw_i
$$
with $f_i\in \hat R_0$. By Lemma \ref{LemmaA22}, we have that
\begin{equation}\label{eqA23}
\hat\nu^*(f)=\min_i\{\hat\nu(f_i)+\hat\nu^*(w_i)\}
\end{equation} 
and there is a unique value of $i$ giving this value. In particular, the classes
$$
[w_i]:= {\rm in}_{\nu^*}(w_i)\in \mathcal P_{\hat\nu^*(w_i)}(\hat T)/\mathcal P^+_{\hat\nu^*(w_i)}(\hat T)
$$
for $1\le i\le e$ generate
$$
G_2:= \bigoplus_{\gamma\in \Gamma_{\nu^*}}\mathcal P_{\gamma}(\hat T)/\mathcal P_{\gamma}^+(\hat T)
$$
as a
$$
G_1 := \bigoplus_{\gamma\in \Gamma_{\nu}}\mathcal P_{\gamma}(\hat R_0)/\mathcal P_{\gamma}^+(\hat R_0)
$$
module. $[w_1],\ldots, [w_e]$ are a free basis of $G_2$ over $G_1$ since $\hat\nu^*(w_i)-\hat\nu^*(w_j)\not \in \Gamma_{\nu}$ if $i\ne j$.

It follows from  Theorem 4.10 and Lemma 4.4 \cite{CP} that $E_1/F_1$ is Galois with Galois group 
$G(E_1/F_1)\cong\ZZ^n/A\ZZ^n$. We summarize the construction given there of a natural isomorphism. 
Let $\omega$ be a primitive $e$-th root of unity in $k$. An element $\sigma\in G(E_1/F_1)$ is determined by its action on $z_i$ for $1\le i\le s$. Since $z_i^e\in \hat R$ for $1\le i\le s$ we have that $\sigma(z_i)=\omega^{c_i}z_i$ for some $c_i\in \ZZ$. From the relation $x_i=z_1^{a_{i1}}\cdots z_s^{a_{is}}$
for $1\le i\le s$, we have that
$$
G(E_1/F_1)=\{c\in \ZZ^s\mid Ac\in e\ZZ^s\}/e\ZZ^s
$$
and the natural map
\begin{equation}\label{eqA40}
\Psi:\ZZ^s/A\ZZ^s\rightarrow G(E_1/F_1)
\end{equation}
defined by $\Psi(c)=({\rm adj }A) c$ is an isomorphism. 
We have shown that  $E_0$ is Galois over $F_0$ (shown before (\ref{eqA30})) and $[E_0:F_0]=e$ by (\ref{eqA9}). Suppose $\sigma\in G(E_1/F_1)$. Since $E_0$ is Galois over $F_0$, we have that $\sigma|E_0\in G(E_0/F_0)$. Suppose $\sigma|E_0= {\rm id}$. $E_0=F_0(K^*)$ then implies $\sigma|K^*={\rm id}$, and $E_1=F_1(K^*)$ implies $\sigma={\rm id}$. Now
$|G(E_1/F_1)|=|G(E_0/F_0)|=e$ implies the restriction map $G(E_1/F_1)\rightarrow G(E_0/F_0)$ is an isomorphism.
Comparing with the isomorphism (\ref{eqA40}) and (\ref{eqA11}) we have a natural group isomorphism
$$
\Gamma_{\nu^*}/\Gamma_{\nu}\cong \ZZ^s/A^t\ZZ^s\cong \ZZ^s/A\ZZ^s\cong G(E_0/F_0).
$$
 This gives us a natural ``diagonal'' action of $\Gamma_{\nu}/\Gamma_{\nu^*}$ on 
$\hat T=\bigoplus_{i=1}^e\hat R_0w_i$ and ${\rm gr}_{\nu^*}(T)=\bigoplus_{i=1}^e{\rm gr}_{\nu}(R_0)[w_i]$ such that
$$
(\hat T)^{\Gamma_{\nu}/\Gamma_{\nu^*}}\cong \hat R_0
$$
and
$$
{\rm gr}_{\nu^*}(T)^{\Gamma_{\nu^*}/\Gamma_{\nu}}\cong {\rm gr}_{\nu}(R_0),
$$
completing the proof of Theorem \ref{Theorem2}.

\end{document}